\documentclass[a4paper,11pt,twoside]{amsart}

\usepackage[utf8]{inputenc}

\usepackage{amsmath, amsthm, amsfonts, amssymb}

\usepackage[colorlinks=true]{hyperref}

\usepackage{mathabx}

\usepackage{comment}

\newcommand{\R}{\ensuremath{\mathbb{R}}}

\newcommand{\N}{\ensuremath{\mathbb{N}}}

\renewcommand{\div}{\textrm{div }}
\newcommand{\curl}{\textrm{curl }}

\theoremstyle{definition} 
\newtheorem{theorem}{Theorem}[section]
\newtheorem{proposition}[theorem]{Proposition}
\newtheorem{definition}[theorem]{Definition}

\newtheorem{lemma}[theorem]{Lemma}
\newtheorem{remark}[theorem]{Remark}

\begin{document}

\title[Variable mixing speed for Muskat]{Mixing solutions for the Muskat problem with variable speed}

\author{Florent Noisette}
\address{{\'E}cole Normale Sup{\'e}rieure, 45 rue d'Ulm, 75230 Paris, France}
\email{fnoisette@clipper.ens.fr}

\author{L\'aszl\'o Sz\'ekelyhidi Jr.}
\address{Institut f\"ur Mathematik, Universit\"at Leipzig, D-04103 Leipzig, Germany}
\email{laszlo.szekelyhidi@math.uni-leipzig.de}

\date{\today}

\begin{abstract}
We provide a quick proof of the existence of mixing weak solutions for the Muskat problem with variable mixing speed. Our proof is considerably shorter and extends previous results in \cite{ccf:ipm} and \cite{fsz:ipm}.
\end{abstract}

\maketitle

\section{Introduction}

The mathematical model for the evolution of two incompressible fluids moving in a porous medium, such as oil and water in sand, was introduced by Morris Muskat in his treatise \cite{mus:1937}, and is based on Darcy's law (see also \cite{Saffman:1958cu,wooding}). In this paper we focus on the case of constant permeability under the action of gravity so that, after non-dimensionalizing, the equations describing the evolution of density $\rho$ and velocity $u$ are given by (see \cite{otto:ipm,cg:condynipm} and references therein)
\begin{align}
\partial_t \rho + \div(\rho u) &= 0\,,\label{IPM1} \\
\div u &= 0\,,\label{IPM2}\\
u + \nabla p &= -(0,\rho)\,,\label{IPM3} \\
\rho(x,0) &= \rho_0(x)\,.\label{IPM4}
\end{align}
We assume that at the initial time the two fluids, with densities $\rho^+$ and $\rho^-$, are separated by an interface which can be written as the graph of a function over the horizontal axis. That is, 
\begin{align}\label{init}
\rho_0(x) = \begin{cases} \rho^+ & x_2 > z_0(x_1),\\ \rho^- & x_2 < z_0(x_1). \end{cases}
\end{align}
Thus, the interface separating the two fluids at the initial time is given by $\Gamma_0 := \{(s,z_0(s))| s \in \R\}$. Assuming that $\rho(x,t)$ remains in the form \eqref{init} for positive times, the system reduces to a non-local evolution problem for the interface $\Gamma$. If the sheet can be presented as a graph as above, one can show (see for example \cite{cg:condynipm}) that the equation for $z(s,t)$ is given by
\begin{align}\label{evolz}
\partial_t z(s,t) = \frac{\rho^- - \rho^+}{2\pi} \int_{- \infty}^{\infty}  \frac{(\partial_{s} z(s,t) - \partial_{s} z(\xi,t))(s-\xi)}{(s-\xi)^2 + (z(s,t) - z(\xi,t))^2} d\xi.
\end{align}
Linearising \eqref{evolz} around the flat interface $z=0$ reduces to $\partial_tf=\frac{\rho^+-\rho^-}{2}\mathcal{H}(\partial_sf)$, where $\mathcal{H}$ denotes the classical Hilbert transform. Thus one distinguishes the following cases: The case $\rho^+ > \rho^-$ is called the unstable regime and amounts to the situation where the heavier fluid is on top. The case $\rho^+< \rho^-$ is called the stable regime. In the stable case, this equation is locally well-posed in $H^3(\R)$, see \cite{cg:condynipm,cgvs:gr2dm}, whereas in the unstable case, we have an ill-posed problem, see \cite{Saffman:1958cu,cg:condynipm}, and there are no general existence results for \eqref{evolz} known. Thus, the description of \eqref{IPM1}-\eqref{IPM4} as a free boundary problem seems not suitable for the unstable regime. Indeed, as shown in experiments \cite{wooding}, in this regime the sharp interface seems to break down and the two fluids start to mix on a mesoscopic scale. In a number of applications \cite{mus:1937, wooding}, however, it is precisely this mixing process in the unstable regime which turns out to be highly relevant, calling for an amenable mathematical framework.   

\subsection*{Mixing solutions and admissible subsolutions}

A notion of solution, which allows for a meaningful existence theory and at the same time able to represent the physical features of the problem such as mixing, was introduced in \cite{sze:ipm}; it is based on the concept of subsolution, which appears naturally when considering stability of the nonlinear system \eqref{IPM1}-\eqref{IPM4} under weak convergence \cite{dls:hprinc}. This point of view was pioneered by L.~Tartar in the 1970s-80s in his study of compensated compactness \cite{Tartar:1977}, and experienced renewed interest in the past 10 years in connection with the theory of convex integration, applied to weak solutions in fluid mechanics \cite{dlseuler,dls:admcrit,cfg:luipm,shvydkoy:2011}. In order to state the definition, we recall that after applying a simple affine change of variables we may assume $|\rho^\pm|=1$. In particular, for the rest of the paper we will be concerned with the unstable case, so that 
$$
\rho^+=+1,\quad \rho^-=-1.
$$
With this normalization subsolutions are defined as follows (c.f.~\cite[Definition 4.1]{sze:ipm}).

\begin{definition}\label{d:subsol}
Let $T > 0$. We call a triple $(\rho,u,m) \in L^{\infty}(\R^2 \times [0,T))$ an admissible subsolution of \eqref{IPM1}-\eqref{IPM4} if there exist open domains $\Omega^{\pm}, \Omega_{mix}$ with $\overline{\Omega^+} \cup \overline{\Omega^-} \cup \Omega_{mix} = \R^2\times [0,T)$ such that
\begin{enumerate}
\item[(i)] The system
\begin{equation}\label{IPMsub}
\begin{split}
\partial_t \rho + \div m &= 0\\
\div  u &= 0\\
\curl  u &= -\partial_{x_1} \rho\\
\rho\vert_{t=0} &= \rho_0
\end{split}
\end{equation}
holds in the sense of distributions in $\R^2 \times [0,T)$;
\item[(ii)] The pointwise inequality
\begin{align}\label{str}
\left|m-\rho u + \frac{1}{2} (0,1-\rho^2) \right| \leq  \frac{1}{2} \left(1-\rho^2 \right),
\end{align}
holds almost everywhere;
\item[(iii)] $|\rho(x,t)| = 1$ in $\Omega^{+}\cup\Omega^-$;
\item[(iv)] In $\Omega_{mix}$ the triple $(\rho,u,m)$ is continuous and \eqref{str} holds with a strict inequality.
\end{enumerate}
\end{definition}

Observe that whenever $\rho=\pm 1$ for an admissible subsolution, then by \eqref{str} the system \eqref{IPMsub} reduces to \eqref{IPM1}-\eqref{IPM3}. Conversely, in the set $\Omega_{mix}$ the subsolution $\rho$ represents the coarse-grained density of a microscopically mixed state. More precisely, we have the following theorem from \cite{sze:ipm,ccf:ipm}:
\begin{theorem}\label{t:mixing}
Suppose there exists an admissible subsolution $(\bar{\rho},\bar{u},\bar{m})$ to \eqref{IPM1}-\eqref{IPM4} . Then there exist infinitely many 
admissible weak solutions $(\rho,u)$ with the following additional mixing property: For any $r>0$, $0<t_0<T$ and $x_0\in\R^2$ such that $B:=B_r(x_0,t_0)\subset \Omega_{mix}$, both sets $\{(x,t)\in B:\,\rho(x,t)=\pm 1\}$ have strictly positive Lebesgue measure. 

Furthermore, there exists a sequence of such admissible weak solutions $(\rho_k,u_k)$ such that $\rho_k\overset{*}\rightharpoonup \bar{\rho}$ as $k\to\infty$. 
\end{theorem}

In other words $\bar{\rho}$ represents a sort of ``coarse-grained, average'' density.
Recently this coarse-graining property of Theorem \ref{t:mixing} was sharpened in \cite{cfm:degraded} along the lines of \cite{dls:admcrit} to the statement that, essentially, $\bar\rho$ denotes the average density not just on space-time balls $B$ but also space-balls for {\em every} time $t>0$.
 
Theorem \ref{t:mixing} is based on a general and very robust Baire-category type argument (c.f.~\cite{dls:hprinc} as well as \cite[Appendix]{sze:ipm}), and basically highlights the key observation that {\em a central object in the study of unstable hydrodynamic interfaces is a suitably defined subsolution}. In recent years this approach has been successfully applied in various contexts: for the incompressible Euler system in the presence of a Kelvin-Helmholtz instability \cite{sze:vortex,bszw:nonuniq,mengualsz}, the density-driven Rayleigh-Taylor instability \cite{gksz:2020}, the Muskat problem with fluids of different mobilities \cite{mengual:2020}, as well as in the context of the compressible Euler system \cite{cdk:gip}.

\subsection*{Evolution of the mixing region - the pseudointerface}

An interesting phenomenon concerning the evolution of the coarse-grained interface was discovered by A.~Castro, D.~Cordoba and D.~Faraco in \cite{ccf:ipm}: for general (sufficiently smooth) initial curves $\Gamma_0$ the mixing (sub)solutions exhibit a two-scale dynamics. On a fast scale the sharp interface diffuses to a mixing zone $\Omega_{mix}$ at some speed $c>0$, which has a stabilizing effect on the overall dynamics. On a slower scale the mixing zone itself begins to twist and evolve according to the now regularized evolution of the mid-section of $\Omega_{mix}$, called a pseudo-interface. 

The authors in \cite{ccf:ipm} showed that appearance of a mixing zone with speed $c$ is compatible with the requirements of Definition \ref{d:subsol} provided $c<2$ (for the flat initial condition $c<2$ was also the upper bound reached in \cite{sze:ipm}, in agreement with the relaxation approach in \cite{otto:ipm}), and by a suitable {\em ansatz} exhibited the regularized evolution of the pseudo-interface as a nonlinear and nonlocal evolution equation of the form (see (1.11)-(1.12) in \cite{ccf:ipm})
\begin{equation}\label{e:ccf}
\partial_tz=\mathcal{F}(z).
\end{equation}
In a technical tour de force they were able to show well-posedness of \eqref{e:ccf} for initial data $z_0\in H^5(\R)$. Roughly speaking, the key point is that the linearization of \eqref{evolz}, in Fourier space written as $\partial_t\hat f=|\xi|\hat{f}$, is modified by the appearance of the mixing zone to
\begin{equation}\label{e:evolf}
\partial_t\hat{f}=\frac{|\xi|}{1+ct|\xi|}\hat f,
\end{equation}
which leads to $\hat{f}=(1+ct|\xi|)^{1/c}\hat f_0$. The analysis of this equation was performed for constant $c=1$ in \cite{ccf:ipm}, and recently extended to  variable $c=c(s)$ in \cite{acf:2020} (in which case \eqref{e:evolf} has to be interpreted as a pseudodifferential equation) under certain restrictive conditions. In particular, the analysis in \cite{acf:2020} applies under the condition that the range of mixing speed   $c_1\leq c(s)\leq c_2<2$ is small: $0<c_1\leq c_2\leq \frac{c_1}{1-c_1}$; furthermore, high regularity is required: $c\in W^{k,\infty}(\R)$ with $k>c^{-1}$  (c.f.~\cite[Definition 4]{acf:2020}).

It was observed in \cite{fsz:ipm} that the ansatz of \cite{ccf:ipm} is too restrictive since after a short initial time the macroscopic evolution of the pseudo-interface quite likely becomes non-universal. Thus, the authors in \cite{fsz:ipm} replaced the equation \eqref{e:ccf} by a simple expansion in time upto second order,
\begin{equation}\label{e:z0z1z2}
z(s,t)=z_0(s)+tz_1(s)+\tfrac12t^2z_2(s),
\end{equation}
and showed that a suitable choice of $z_1$ and $z_2$ leads to an evolution which is compatible with Definition \ref{d:subsol} for any constant speed $c\in (0,2)$. More precisely, in the expansion \eqref{e:z0z1z2} $z_1=\partial_tz|_{t=0}:=u_{\nu}|_{t=0}$ is chosen as the normal velocity induced on the interface by \eqref{evolz} at time $t=0$, whereas $z_2$ involves a non-local operator of the same type applied to $z_0$ plus a local curvature term: 
$$
z_2:=T[z_0]+c\frac{1-(\partial_sz_0)^2}{(1+(\partial_sz_0)^2)^{1/2}}\kappa_0,
$$  
where $\kappa_0=\kappa_0(s)$ is the curvature of the initial interface $\Gamma_0$. This expansion reveals an important difference to the approach in \cite{ccf:ipm}: the regularity of the pseudointerface does not deteriorate with small $c\to 0$, in sharp contrast to the evolution in \eqref{e:evolf}. From a physical point of view this is natural to expect, if one takes into account the scale separation in the two dynamics: once a mixing zone appears, the pseudointerface is a matter of arbitrary choice, the only relevant object for the system \eqref{IPMsub}-\eqref{str} being the set $\Omega_{mix}$. Thus, a coupling between the two dynamics should appear at most in terms of higher order fluctuations; indeed, a closer look \cite{ccf:ipm,acf:2020} reveals that the deterioration of regularity observed in \eqref{e:evolf} in fact applies to $f=\partial_s^4z$. 

Motivated by this heuristic, in this short note we extend and simplify the analysis of \cite{sze:ipm} by 
\begin{enumerate}
\item allowing for variable mixing speed $c=c(s)$ within the whole range $0<\inf_{\R}c\leq \sup_{\R}c<2$, with no degeneration of regularity;
\item allowing for asymptotically vanishing mixing speed $c(s)\to 0$ as $|s|\to\infty$ in case the initial interface $z_0$ is asymptotically horizontal. 
\end{enumerate}
Moreover, our analysis shows that the expansion \eqref{e:z0z1z2} above, obtained in \cite{sze:ipm}, remains valid upto second order even in this generality, thus giving further evidence towards {\em universality} of the macroscopic evolution.

\subsection*{The main result}

In this section we state the precise form of our main result. 

Our assumption on the initial datum is that the initial interface is asymptotically flat with some given slope $\beta\in\R$, i.e. $\rho_0$ is given by \eqref{init} with
\begin{equation}\label{e:z0}
	z_0(s) = \beta s + \tilde{z}_0(s)
\end{equation}
for some $\tilde{z}_0$ with sufficiently fast decay at infinity, using -- as in  \cite{sze:ipm} -- the following weighted H\"older norms: for any $0<\alpha<1$ set
$$
\|f\|_0^*:=\sup_{s\in\R}(1+|s|^{1+\alpha})|f(s)|.
$$
Furthermore, we define the associated H\"older (semi-)norms as follows. We set
$$
[f]^*_{\alpha}:=\sup_{|\xi|\leq 1,s\in\R}(1+|s|^{1+\alpha})\frac{|f(s-\xi)-f(s)|}{|\xi|^{\alpha}},
$$
and for any $k\in\N$ 
$$
\|f\|_{k,\alpha}^*:=\sup_{s\in\R,j\leq k}(1+|s|^{1+\alpha})|\partial_s^jf(s)|+[\partial_s^kf]_{\alpha}^*.
$$
We denote by $C_*^{k,\alpha}(\R):=\{f\in C^{k,\alpha}(\R):\,\|f\|_{k,\alpha}^*<\infty\}$.

Next, we describe the geometry of the coarse-grained evolution. 
Given a pseudointerface $z:\R \times [0,T] \rightarrow \R$ and mixing speed $c:\R\to(0,\infty)$ define $\Omega^{\pm}(t)$ and $\Omega_{mix}(t)$ as 
\begin{equation}\label{omega}
\begin{split}
\Omega^+(t) &= \{x \in \R^2 | x_2 > z(x_1,t) + c(x_1)t\},\\
\Omega_{mix}(t) &= \{x \in \R^2 |z(x_1,t) - c(x_1)t < x_2 < z(x_1,t) + c(x_1)t\},\\
\Omega^-(t) &= \{x \in \R^2 | x_2 < z(x_1,t) - c(x_1)t\},\\
\end{split}
\end{equation}
and set
$$
\Omega^{\pm}=\bigcup_{t>0}\Omega^{\pm}(t),\quad \Omega_{mix}=\bigcup_{t>0}\Omega_{mix}(t).
$$

\begin{theorem}\label{Mth}
Let $z_0(s) = \beta s + \overline{z}_0(s)$ with $\overline{z}_0 \in C^{3,\alpha}_*(\R)$ for some $0<\alpha<1$ and $\beta \in \R$. Let $c=c(s)>0$ with $\sup_sc(s)<2$ and $\partial_sc\in C^{\alpha}_*(\R)$. If $\inf_sc(s)=0$, assume in addition that $\beta=0$ and there exists $c_{min}>0$ such that
$$
c(s)\geq c_{min}(1+|s|^{2\alpha/3})^{-1}.
$$
Then there exists $T>0$ such that there exists a pseudo-interface $z\in C^2([0,T];C^{1,\alpha}(\R))$ with $z|_{t=0}=z_0$ for which the mixing zone defined in \eqref{omega} admits admissible subsolutions on $[0,T]$. In particular there exist infinitely many admissible weak solutions to \eqref{IPM1}-\eqref{IPM4} on $[0,T]$ with mixing zone given by \eqref{omega}.
\end{theorem}

Observe that under the conditions in the theorem the function $c$ has limits at infinity $s\to\pm\infty$.

The paper is organised as follows. In Section \ref{s:sub} we show that an admissible subsolution exists provided certain smallness conditions are satisfied on the temporal expansion of the pseudo-interface - see Proposition \ref{p:SubsN}. This section closely follows the construction in \cite{sze:ipm}, in particular the construction of symmetric piecewise constant densities in \cite[Section 5]{sze:ipm}. 

Then in Section \ref{s:u} we obtain a regular expansion in time $t$ for the normal component of the velocity across interfaces for arbitrary mixing speeds. Our key result in this section is Proposition \ref{p:expansion}, see also  Remark \ref{r:expansion} for a simplified statement. It is worth pointing out that validity of the expansion requires minimal smoothness assumptions on the pseudo-interface and, at variance with the approach in \cite{acf:2020}, does not degenerate as $c\to 0$ or $c\to 2$. Finally, in Section \ref{s:z} we complete the proof of Theorem \ref{Mth}. 

We remark in passing that if $\beta\neq 0$, the statement of the theorem continues to hold provided the lower bound on $c(s)$ is strengthened to 
$$
c(s)\geq c_{min}(1+|s|^{\alpha/2})^{-1}.
$$
The proof of this requires minor modifications in Proposition \ref{p:SubsN}, in particular replacing the term $c^{3/2}$ in \eqref{z2N} by $c^2$. As such modifications unnecessarily complicate the presentation without added value, we chose not to include the details here.

\subsection*{Acknowledgments} 
This project has received funding from the European Research Council (ERC) under the European Union’s Horizon 2020 research and innovation programme (Grant Agreement No. 724298).
This work was initiated during the visit of the authors to the Hausdorff Research Institute for Mathematics (HIM), University of Bonn, in March 2019. This visit was supported by the HIM. Both the support and the hospitality of HIM are gratefully acknowledged.


\section{Subsolutions with variable mixing speed}\label{s:sub}

We start by fixing $N\in \N$ and setting
\begin{equation}\label{e:ci}
c_i(s)=\frac{2i-1}{2N-1}c(s)\quad\textrm{ for }i=1,\dots,N.
\end{equation}
Define the density $\rho(x,t)$ to be the piecewise constant function
\begin{align}\label{e:densN}
\rho(x,t) = \begin{cases} 1 & x \in \Omega^+(t),\\ \frac{i}{N} & c_i(x_1)t<x_2-z(x_1,t)<c_{i+1}(x_1)t\\ 0& -c_1(x_1)t<x_2-z(x_1,t)<c_1(x_1)t\\ -\frac{i}{N} & c_i(x_1)t<z(x_1,t)-x_2<c_{i+1}(x_1)t\\-1 & x \in \Omega^-(t), \end{cases}
\end{align}
with $i=1,\dots,N-1$. This definition of $\rho$ already determines the velocity $u$ by the kinematic part of \eqref{IPMsub}, namely the Biot-Savart law (see Section \ref{s:u} below)
\begin{equation}\label{e:BS}
\begin{split}
\div u&=0,\\
\curl u&=-\partial_{x_1}\rho.
\end{split}
\end{equation}
Observe that $\rho$ is piecewise constant, with jump discontinuities across $2N$ interfaces 
\begin{equation}\label{e:Gammapm}
\Gamma^{(\pm i)}(t)=\left\{ (s,z^{(\pm i)}(s,t)):\,s\in\R\right\},\quad z^{(\pm i)}(s,t)=z(s,t)\pm c_i(s)t.
\end{equation}
It is well known \cite{cg:condynipm} that, provided the interfaces are sufficiently regular, the solution $u$ to \eqref{e:BS} is then globally bounded, smooth in $\R^2\setminus\bigcup_{i}\Gamma^{(i)}$ with well-defined traces on $\Gamma^{\pm}$, and the normal component 
\begin{equation}\label{e:unuN}
u_{\nu}^{(i)}(s,t):=u(s,z^{(i)}(s,t),t)\cdot \begin{pmatrix}-\partial_{s}z^{(i)}(s,t)\\ 1\end{pmatrix}
\end{equation}
is continuous across the interfaces $\Gamma^{(i)}$ for $i=\pm 1,\dots,\pm N$.

Indeed, we will see in the next section that 
this is the case provided 
\begin{equation}\label{e:zc}
z(s,t)=\beta s+\tilde{z}(s,t),
\end{equation}
with $\tilde z(\cdot,t) \in C^{1,\alpha}_*(\R)$, $\partial_sc\in C^{\alpha'}_*(\R)$ and some $0<\alpha,\alpha'<1$ and $\beta\in \R$.

Our main result in this section is as follows:

\begin{proposition}\label{p:SubsN}
Let $z(s,t)$ and $c(s)$ be as in \eqref{e:zc} with $\tilde z\in C^1([0,T];C^{1,\alpha}_*(\R))$, $\partial_sc\in C^{\alpha'}_*(\R)$ and $0<c(s)\leq c_{max}$ $\forall s$ for some $c_{max}<\frac{2N-1}{N}$. Let $\rho$ be defined by \eqref{e:densN} and $u$ the corresponding velocity field $u$ according to \eqref{e:BS}, with normal traces as in \eqref{e:unuN}. Assume that
\begin{align}
&\lim_{t\to 0}\sup_{s}\frac{1}{c(s)}\left|\partial_tz(s,t)-u_{\nu}^{(\pm i)}(s,t)\right|=  0\,\textrm{ for all }i\,, \label{z1N}\\
&\lim_{t\to 0}\sup_s\frac{1}{tc^{3/2}(s)}\left|\int_0^s\partial_tz(s',t)-\sum_{i=1}^N\frac{u_\nu^{(+i)}(s',t)+u_{\nu}^{(-i)}(s',t)}{2N}\,ds'\right| =0\,, \label{z2N}
\end{align} 
and furthermore, there exists $M>0$ such that
\begin{equation}\label{z3N}
|\partial_s z(s,t)|,|\partial_s c(s)|\leq Mc^{1/2}(s)\quad\forall\,s\in\R, t\in[0,T].
\end{equation}
Then there exists $T'\in (0,T]$ and a vectorfield $m:\R^2\times(0,T')\to\R^2$ such that $(\rho,u,m)$ is an admissible subsolution on $[0,T')$ with mixing zone $\Omega_{mix}$ given by \eqref{omega}.
\end{proposition}

\begin{proof}
The proof follows closely the proof of Theorem 5.1 in \cite{sze:ipm}, adapted here to the variable speed setting.
Set
\begin{align*}
m = \rho u - (1-\rho^2) (\gamma + \tfrac{1}{2}e_2)
\end{align*}
for some $\gamma=\gamma(x,t)$, with $\gamma\equiv 0$ in $\Omega^{\pm}$. Then \eqref{str} amounts to the condition
$$
|\gamma| < \frac{1}{2}\quad\textrm{ in }\Omega_{mix},
$$
whereas \eqref{IPMsub} is equivalent to $\div\gamma=0$ in $\Omega_{mix}\setminus\bigcup_{i=1}^N(\Gamma^{i}\cup\Gamma^{-i})$ together with $2(N-1)$ jump conditions
\begin{equation}\label{e:jump}
[\rho]_{\Gamma^{(i)}}\partial_tz^{(i)}+[m]_{\Gamma^{(i)}}\cdot \begin{pmatrix}\partial_{x_1}z^{(i)}\\ -1\end{pmatrix}=0\quad\textrm{ on }\Gamma^{(i)}\end{equation}
for $i=\pm 1,\dots,\pm N$,
where $[\cdot]_{\Gamma^{(i)}}$ denotes the jump on $\Gamma^{(i)}$. Observe that 
$$
\Omega_{mix}\setminus\bigcup_{i=1}^N(\Gamma^{i}\cup\Gamma^{-i})=\bigcup_{i=-N+1}^{N-1}\Omega^{(i)},
$$
with connected open sets $\Omega^{(i)}$ defined as
\begin{equation}\label{e:densni}
\begin{split}
\Omega^{(i)} &= \{z(x_1,t)+c_{i}(x_1)t < x_2 < z(x_1,t)+c_{i+1}(x_1)t\},\\
\Omega^{(-i)} &= \{z(x_1,t)-c_{i+1}(x_1)t < x_2 < z(x_1,t)-c_{i}(x_1)t\},
\end{split}	
\end{equation}
for $i=1,\dots,N-1$ and 
$$
\Omega^{(0)}= \{z(x_1,t)-c_{1}(x_1)t < x_2 < z(x_1,t)+c_{1}(x_1)t\}
$$
The divergence-free condition is then taken care by setting 
$$
\gamma=\nabla^\perp g^{(i)}\,\textrm{  in } \Omega^{(i)},\quad i=-N\dots N,
$$
where $g^{(N)}=g^{{(-N)}}=0$ and $g^{(i)}\in C^1(\overline{\Omega^i})$ for $i=-(N-1)\dots (N-1)$ are to be determined. 
Then, \eqref{str} amounts to the conditions
\begin{equation}\label{e:nablagi}
|\nabla g^{(i)}|<\frac12\quad\textrm{ in }\Omega^i\quad i=-(N-1)\dots(N-1),
\end{equation}
and \eqref{e:jump} reduces to conditions on the tangential derivatives on each interface: for any $i=1,\dots,N$ we require on $\Gamma^{(\pm i)}$
\begin{align}
\partial_\tau g^{(\pm(i-1))}&=\frac{1}{1-(\frac{i-1}{N})^2}\left\{\frac{c_i}{N}-\frac{2i-1}{2N^2}+\left(1-(\tfrac{i}{N})^2\right)\partial_{\tau}g^{(\pm i)}\pm \frac{\partial_tz-u_\nu^{(\pm i)}}{N}\right\}\notag\\
&=h^{(\pm i)}+\frac{1-(\tfrac{i}{N})^2}{1-(\frac{i-1}{N})^2}\partial_{\tau} g^{(\pm i)}\quad\textrm{ on }\Gamma^{\pm i}\,,\label{e:defh}
\end{align}
with
$$
h^{(\pm i)}(s,t):=\frac{1}{1-(\frac{i-1}{N})^2}\left\{\frac{c_i(s)}{N}-\frac{2i-1}{2N^2}\pm \frac{\partial_tz(s,t)-u_\nu^{(\pm i)}(s,t)}{N}\right\}
$$
and 
$$
\partial_\tau g^{(i)}\big\vert_{\Gamma^{(j)}}(s,t):=\partial_{s}\bigl[g^{(i)}(s,z^{(j)}(s,t),t)\bigr].
$$
Next, for $i\neq 0$ we make the choice that $g^{(i)}$ is a function of $x_1,t$ only. Then, using the fact that $g^{(\pm N)}=0$, we may use \eqref{e:defh} to inductively define $g^{(\pm (N-1))}$, $g^{(\pm (N-2))},\dots,g^{(\pm 1)}$ as
\begin{align}
g^{(\pm i)}(x_1,t)&=\int_0^{x_1} h^{(\pm (i+1))}(s',t)+\frac{1-(\tfrac{i+1}{N})^2}{1-(\frac{i}{N})^2}\partial_{\tau} g^{(\pm (i+1))}(s',t)\,ds'.\label{e:defgi}
\end{align}
In particular we obtain
\begin{equation*}
	\begin{split}
		\partial_{x_1}g^{(\pm i)}&=\frac{1}{1-(\tfrac{i}{N})^2}\sum_{j=i+1}^N\left(\frac{c_j}{N}-\frac{2j-1}{2N^2}\pm \frac{\partial_t z-u_{\nu}^{(\pm j)}}{N}\right)\\
		&=\frac{N}{2N-1}\left(c-\frac{2N-1}{2N}\right)\pm \frac{N}{N^2-i^2}\sum_{j=i+1}^N[\partial_tz-u_{\nu}^{(j)}]\\
		&=-\frac{1}{2}+c(s)\left(\frac{N}{2N-1}+o(1)\right),
	\end{split}
\end{equation*}
where $o(1)$ denotes terms going to zero uniformly in $s$ as $t\to 0$ and we have used \eqref{z1N} in the last line. Since we also set $\partial_{x_2}g^{(i)}=0$, and $0<c(s)\leq c_{max}< \tfrac{2N-1}{N}$, we can deduce that \eqref{e:nablagi} holds for sufficiently small $t>0$. 

Next, we turn our attention to $g^{(0)}$ on $\Omega^{(0)}$. 
For $s\in\R$, $t\in (0,T)$ and $\lambda\in [-1,1]$ define
\begin{equation}\label{e:hatg}
\hat{g}(s,\lambda,t):=g^{(0)}(s,z(s,t)+\lambda c_1(s)t,t).
\end{equation}
In order to satisfy \eqref{e:defh} we set
$$
\hat{g}(s,\pm 1,t):=\int_{0}^s h^{(\pm 1)}+\left(1-\frac{1}{N^2}\right)\partial_{x_1}g^{(\pm 1)}\,ds',
$$
and, more generally, for $\lambda\in [-1,1]$ 
\begin{equation*}
\hat{g}(s,\lambda,t):=\frac{1+\lambda}{2}\hat{g}(s,1,t)+\frac{1-\lambda}{2}\hat{g}(s,-1,t).
\end{equation*}
Then
\begin{align*}
\partial_{\lambda}\hat{g}&=\int_0^s\sum_{j=1}^N\left(\partial_tz-\frac{u_\nu^{(+j)}+u_{\nu}^{(-j)}}{2}\right)\,ds'\\
&=tc^{3/2}(s)o(1),\\
\partial_s\hat{g}&=-\frac{1}{2}+\frac{1}{N}\sum_{j=1}^Nc_j+\frac{\lambda+1}{2N}\sum_{j=1}^N[\partial_t z - u_\nu^{(+j)}] + \frac{\lambda-1}{2N} \sum_{j=1}^N[\partial_t z - u_\nu^{(-j)}]\\
&=-\frac{1}{2}+c(s)\left(\frac{N}{2N-1}+o(1)\right)
\end{align*}
where we have used \eqref{z1N}-\eqref{z2N}. 
By differentiating \eqref{e:hatg} with respect to $s$ and $\lambda$ and using \eqref{z1N}-\eqref{z2N} we deduce
\begin{equation}\label{e:gradg}
\begin{split}
\partial_{x_2}g^{(0)}&=c^{1/2}(s)o(1)\\
\partial_{x_1}g^{(0)}&=-\frac{1}{2}+c(s)(\tfrac{N}{2N-1}+o(1))+c^{1/2}(s)(|\partial_{s}z|+t|\partial_{s}c|)o(1)\\
&=-\frac{1}{2}+c(s)(\tfrac{N}{2N-1}+o(1))
\end{split}
\end{equation}
where we have used \eqref{z3N} in the last line. Consequently 
\begin{align*}
|\nabla g|^2&=\frac14-c(s)(\tfrac{N}{2N-1}+o(1))+c^2(s)(\tfrac{N}{2N-1}+o(1))^2+c(s)o(1)\\
&=\frac14-c(s)(\tfrac{N}{2N-1}+o(1))+c^2(s)(1+o(1))\\
&\leq \frac14-\tfrac{N}{2N-1}c(s)\bigl(1-\tfrac{N}{2N-1}c_{max}+o(1)\bigr).
\end{align*}
Since $\tfrac{N}{2N-1}c_{max}<1$, we deduce that  
\begin{align*}
|\nabla g^{(0)}| < \frac{1}{2}\quad\textrm{ for sufficiently small }t>0.
\end{align*}
This concludes the proof.

\end{proof}


\section{The velocity $u$}\label{s:u}

In this section we analyse more closely the normal component of the velocity, given in \eqref{e:unuN}, where the velocity $u$ is the solution of the system \eqref{e:BS} with piecewise constant density $\rho$ given in \eqref{e:densN}. Following the computations in \cite{cg:condynipm} and \cite{fsz:ipm} we see that for any $t>0$
\begin{equation}\label{e:unui}
u_{\nu}^{(i)}(s,t)=\sum_{j}\frac{1}{2\pi N}PV\int_{\R}\frac{\partial_sz^{(j)}(s-\xi,t)-\partial_sz^{(i)}(s,t)}{\xi}\Phi_{ij}(\xi,s,t)\,d\xi,
\end{equation}
where the sum is over $j=\pm 1,\dots,\pm N$, the kernels $\Phi_{ij}(\xi,s,t)$ are defined as
\begin{equation}\label{e:defPhiij}
	\Phi_{ij}(\xi,s,t)=\frac{\xi^2}{\xi^2+(z^{(i)}(s,t)-z^{(j)}(s-\xi,t))^2}\,,
\end{equation}
and 
\begin{equation}\label{e:defzj}
z^{(i)}(s,t)=z(s,t)+c_i(s)t. 
\end{equation}
with the convention $c_{-i}(s)=-c_i(s)$, where $c_i(s)$ is defined in \eqref{e:ci} for $i=1,\dots,N$.
The principal value integral here refers to $PV\int_{\R} = \underset{R\rightarrow \infty}{lim}\int_{-R}^R$. Next, we recall the operator $T_\Phi$ from \cite{sze:ipm}, a weighted version of the Hilbert transform, defined for a weight function $\Phi=\Phi(\xi,s)$ as
\begin{equation}\label{e:defT}
    T_{\Phi}(g)(s):=\frac{1}{2\pi}PV\int_{\R}{\frac{\partial_sg(s-\xi)-\partial_sg(s)}{\xi}\Phi(\xi,s)d\xi}.
\end{equation}
Then \eqref{e:unui} can be written as
\begin{equation}\label{e:unui2}
u_{\nu}^{(i)}=\frac{1}{N}\sum_jT_{\Phi_{ij}}z^{(j)}+\frac{t}{N}\sum_{j\neq i}(\partial_sc_i-\partial_sc_j)I_{ij},
\end{equation}
where we set
\begin{equation}\label{e:Iij}
I_{ij}(s,t)=\frac{1}{2\pi}PV\int_{\R}\Phi_{ij}\frac{d\xi}{\xi}.	
\end{equation}
Observe that in $I_{ij}$ for $i\neq j$ it again suffices to consider the principal value integral as above, with regularization as $|\xi|\to \infty$. Nevertheless, for $i=j$ also a principal value regularization at $|\xi|\to 0$ is necessary - see below in Lemmas \ref{l:I1}-\ref{l:I2}.

We next recall the following bound on $T_{\Phi}$ on H\"older-spaces from \cite{fsz:ipm}, where for the weight we use the following norms: first of all we 
assume that $\Phi^{\infty}(s):=\lim_{|\xi|\to\infty}\Phi(\xi,s)$ exists, $\Phi(\cdot,s)\in C^1(\R\setminus\{0\})$, and set
\begin{equation}\label{e:farfieldPhi}
\begin{split}
\bar{\Phi}&=\xi\left(\Phi-\Phi^{\infty}\right),\quad \Phi^{\infty}=\lim_{|\xi|\to\infty}\Phi(\xi,s),\\
\tilde\Phi&=\xi^2\partial_{\xi}\left(\frac{1}{\xi}\Phi\right)=\xi\partial_{\xi}\Phi-\Phi\,.
\end{split}
\end{equation}
We introduce the norms
\begin{align*}
\vvvert \Phi\vvvert_0&:=\sup_{s\in\R,|\xi|\leq 1}|\Phi(\xi,s)|+\sup_{s\in\R,|\xi|>1}(|\bar{\Phi}(\xi,s)|+|\tilde\Phi(\xi,s)|),\\
\vvvert\Phi\vvvert_{k,\alpha}&:=\max_{j\leq k}\vvvert \partial_s^j\Phi\vvvert_0+[\partial_s^k\Phi]_\alpha+\sup_{|\xi|>1}([\partial_s^k\bar{\Phi}(\xi,\cdot)]_\alpha+[\partial_s^k\tilde\Phi(\xi,\cdot)]_\alpha),
\end{align*}
where we use the convention that $\|\Phi(\xi,\cdot)\|$ denotes a norm in the second argument only and $\|\Phi\|$ denotes a norm joint in both variables. In particular the H\"older-continuity of $\partial_s^k\Phi$ in both variables $\xi,s$ is required in the norm $\vvvert\Phi\vvvert_{k,\alpha}$. Accordingly, we define the spaces
\begin{align*}
\mathcal{W}^0&=\{\Phi\in L^{\infty}(\R^2):\,\Phi^\infty\textrm{ and }\partial_\xi\Phi\textrm{ exist, with }\vvvert \Phi\vvvert_0<\infty\},\\
\mathcal{W}^{k,\alpha}&=\{\Phi\in \mathcal{W}^0:\,\vvvert \Phi\vvvert_{k,\alpha}<\infty\}
\end{align*}
Then, the following version of the classical estimate on the Hilbert transform $T_1=\mathcal{H}\nabla$ on H\"older-spaces holds \cite[Theorem 3.1]{fsz:ipm}:

\begin{theorem}\label{t:TPhi}
For any $\alpha>0$, $f\in C^{1,\alpha}_*(\R)$ and $\Phi\in \mathcal{W}^0$ we have
\begin{equation}\label{e:TPhi1}
\|T_{\Phi}(f)\|_0^*\leq C \vvvert\Phi\vvvert_0\|f\|_{1,\alpha}^*\,.
\end{equation}
Moreover, for any $k\in\N$, $f\in C_*^{k+1,\alpha}(\R)$ and $\Phi\in \mathcal{W}^{k,\alpha}$ 
\begin{equation}\label{e:TPhi2}
\|T_{\Phi}(f)\|_{k,\alpha}^*\leq C\vvvert\Phi\vvvert_{k,\alpha}\|f\|_{k+1,\alpha}^*.
\end{equation}
where the constant depends only on $k$ and $\alpha$. 
\end{theorem}

In the following we analyse boundedness and continuity properties of the type of operators \eqref{e:defT} arising in \eqref{e:unui2}. This will ultimately enable us to derive an expansion in time for $t\to 0$ of the normal velocity components $u^{(i)}_\nu$ in \eqref{e:unui}.

\begin{lemma}\label{l:reg}
	Let $z(s,t)=\beta s+\tilde z(s,t)$ with $\tilde z\in C^{l}([0,T];C^{k+1,\alpha}(\R))$ for some $l,k\in\N$, $0<\alpha<1$, $\beta\in\R$ and $T<\infty$, and let 
	$$
	\Phi(\xi,s,t)=\frac{\xi^2}{\xi^2+(z(s,t)-z(s-\xi,t))^2}.
	$$
	Then $\Phi\in C^{l}([0,T];\mathcal{W}^{k,\alpha})$.
\end{lemma}

\begin{proof}
We start by introducing the following notation: for $z=z(s,t)$ define 
\begin{equation}\label{e:Zt}
Z=Z(\xi,s,t)=\frac{z(s,t)-z(s-\xi,t)}{\xi}=\int_0^1\partial_sz(s-\tau\xi,t)\,d\tau,
\end{equation}
and furthermore, let 
\begin{equation}\label{e:defK}
	K(Z)=\frac{1}{1+Z^2}.
\end{equation}
 Since $K\in C^{\infty}(\R)$ with derivatives of any order uniformly bounded on $\R$, and since 
$\Phi=K\circ Z$, it follows easily from the chain rule that, for any $j\leq l$, 
$\partial_t^j\Phi\in C^{k,\alpha}(\R^2)$ with
$$
\sup_t\|\partial_t^j\Phi(\cdot,\cdot,t)\|_{C^{k,\alpha}(\R^2)}<\infty.
$$ 
Concerning the far-field terms, note that $\Phi^{\infty}=\frac{1}{1+\beta^2}$, hence
$$
\bar{\Phi}=\frac{\beta+Z}{(1+\beta^2)(1+Z^2)}(\tilde{z}(s-\xi,t)-\tilde{z}(s,t))=K_{\beta}(Z)(\tilde{z}(s-\xi,t)-\tilde{z}(s,t)),
$$
where $K_{\beta}(x)=\frac{\beta+x}{(1+\beta^2)(1+x^2)}$ is again a function with derivatives of all order uniformly bounded on $\R$. Therefore the chain rule as above, together with the product rule, easily imply that, for any $j\leq l$, 
$\partial_t^j\bar\Phi\in C^{k,\alpha}(\R^2)$ with
$$
\sup_t\|\partial_t^j\bar\Phi(\cdot,\cdot,t)\|_{C^{k,\alpha}(\R^2)}<\infty.
$$ 
Similarly, $\xi\partial_\xi\Phi=K'(Z)\xi\partial_\xi Z$,
where
$$
\xi\partial_\xi Z=\partial_s\tilde{z}(s-\xi,t)+\int_0^1\partial_s\tilde{z}(s-\tau\xi,t)\,d\tau,
$$
so that, once again for any $j\leq l$, 
$\xi\partial_\xi(\partial_t^j\Phi)=\partial_t^j(\xi\partial_\xi\Phi)\in C^{k,\alpha}(\R^2)$ with
$$
\sup_t\|\xi\partial_\xi\partial_t^j\bar\Phi(\cdot,\cdot,t)\|_{C^{k,\alpha}(\R^2)}<\infty.
$$ 
We deduce that $\partial_t^j\Phi\in \mathcal{W}^{k,\alpha}$ with
$$
\sup_{t\in[0,T]}\vvvert\partial_t^j\Phi\vvvert_{k,\alpha}<\infty
$$	
as required.
\end{proof}

Using the notation introduced above in \eqref{e:Zt}-\eqref{e:defK} we can write
\begin{equation}\label{e:ZZ}
\Phi_{ij}=K\left(Z^{(j)}+\frac{c_{ij}t}{\xi}\right),
\end{equation}
where
$$
Z^{(j)}(\xi,s,t)=\frac{z^{(j)}(s,t)-z^{(j)}(s-\xi,t)}{\xi},\quad c_{ij}(s)=c_i(s)-c_j(s).
$$
Observe that $c_{ii}=0$ so that Lemma \ref{l:reg} applies to $\Phi_{ii}$, but the second term requires more care. In the next lemmata we address boundedness and continuity with respect to the functions $z,c$. 

\begin{lemma}\label{l:sing0}
There exists a constant $C>1$ such that the following holds. Let $z(s)=\beta s+\tilde z(s)$ with $\tilde z\in C^1(\R)$, $c\in C(\R)$, and let 
 	$$
	\Phi(\xi,s)=K\left(Z(\xi,s)+\frac{c(s)}{\xi}\right),
	$$
	where $K(Z)=(1+Z^2)^{-1}$ as in \eqref{e:defK} and $Z(\xi,s)=\frac{z(s)-z(s-\xi)}{\xi}$. 
Then $\Phi\in \mathcal{W}^0$, with
\begin{equation}\label{e:W01}
\vvvert \Phi\vvvert_0\leq C(1+\|\partial_s\tilde z\|_0+\|c\|_0).
\end{equation}
Furthermore, if $\Phi_1,\Phi_2$ are defined as above with $\tilde z_1,\tilde z_2\in C^1(\R)$, then
\begin{equation}\label{e:W02}
\vvvert \Phi_1-\Phi_2\vvvert_0\leq C\|\partial_s\tilde z_1-\partial_s\tilde z_2\|_0.
\end{equation}
\end{lemma}

\begin{proof}
Using that $|K|\leq 1$, we deduce $\sup_{\xi,s}|\Phi|\leq 1$. Moreover, since also $|K'|\leq 1$, we have
\begin{equation*}
\begin{split}
	|\bar{\Phi}|&=|\xi|\left|K(Z+\tfrac{c}{\xi})-K(\beta)\right|\leq |\xi||Z-\beta|+|c|\\
	&\leq 2\|\partial_s\tilde z\|_0+\|c\|_0.
\end{split}
\end{equation*}
Similarly, for $|\xi|\leq 1$ we also have
\begin{equation*}
|\xi\partial_\xi\Phi|=\left|K'(Z+\tfrac{c}{\xi})\right|\left|\xi\partial_\xi Z-\tfrac{c}{\xi}\right|\leq 2\|\partial_s\tilde z\|_0+\|c\|_0.	
\end{equation*}
The estimate \eqref{e:W01} follows.

For the Lipschitz bound \eqref{e:W02} we proceed entirely analogously, using the representation
$$
\Phi_1-\Phi_2=\int_0^1K'\left(\tau Z_1+(1-\tau)Z_2+\frac{c}{\xi}\right)d\tau\,(Z_1-Z_2)
$$
and the bound $|Z_1-Z_2|\leq \|\partial_s\tilde z_1-\partial_s\tilde z_2\|_0$.
\end{proof}

Next, we address continuity of the mapping $c\mapsto K(Z+\frac{c}{\xi})$ at the singularity $c=0$.

\begin{lemma}\label{l:sing1}
There exists a constant $C>1$ such that the following holds.
Let $z(s)=\beta s+\tilde z(s)$ with $\tilde z\in C(\R)$, $c\in C(\R)$, and let 
	\begin{equation*}
	\Phi(\xi,s)=K\left(Z(\xi,s)+\frac{c(s)}{\xi}\right)-K\left(Z(\xi,s)\right),
	\end{equation*}
where $K(Z)=(1+Z^2)^{-1}$ as in \eqref{e:defK} and $Z(\xi,s)=\frac{z(s)-z(s-\xi)}{\xi}$.
Then, for any $f\in C^{1,\alpha}_*(\R)$ we have 
\begin{equation}\label{e:sing1}
\|T_{\Phi}f\|_0^*\leq C[\partial_sf]_{\alpha}^*\|c\|_0^{\alpha}.
\end{equation}
\end{lemma}

\begin{proof}
Using the fact that $|K|,|K'|\leq 1$, we have  
$$
\left|\Phi(\xi,s)\right|=\left|K\left(Z+\tfrac{c}{\xi}\right)-K(Z)\right|\leq C\min\left(1,\frac{c(s)}{|\xi|}\right).
$$ 
Next, recalling the definition of $T_{\Phi}f$ from \eqref{e:defT} we have
\begin{equation*}
\begin{split}
|T_{\Phi}f(s)|&\leq C \int_\R \left|\frac{\partial_s f(s-\xi)-\partial_s f(s)}{\xi}\right|\min\left(1,\frac{c(s)}{|\xi|}\right)\,d\xi\\
&\leq C\frac{[\partial_sf]_{\alpha}^*}{1+|s|^{1+\alpha}}\int_\R |\xi|^{\alpha-1}\min\left(1,\frac{c(s)}{|\xi|}\right)\,d\xi\\
&= C\frac{[\partial_sf]_{\alpha}^*}{1+|s|^{1+\alpha}}c^{\alpha}(s).
\end{split}	
\end{equation*}
\end{proof}

\smallskip

\begin{lemma}\label{l:sing2}
There exists a constant $C>1$ such that the following holds.
Let $z(s)=\beta s+\tilde z(s)$ with $\tilde z\in C^{1,\alpha}(\R)$, $c\in C(\R)$, and let 
	\begin{equation*}
	\begin{split}
	\Phi(\xi,s)=K\left(Z(\xi,s)+\frac{c(s)}{\xi}\right)+K\left(Z(\xi,s)-\frac{c(s)}{\xi}\right)-2K(Z(\xi,s)),
	\end{split}
	\end{equation*}
where $K(Z)=(1+Z^2)^{-1}$ as in \eqref{e:defK} and $Z(\xi,s)=\frac{z(s)-z(s-\xi)}{\xi}$.
Then, for any $f\in C^{2,\alpha}_*(\R)$ and any $s\in \R$ we have 
\begin{equation}\label{e:sing2}
(1+|s|^{1+\alpha})|T_{\Phi}f-\tfrac{|c|}{2}\sigma(\partial_sz)\partial_s^2f|\leq C[f]_{2,\alpha}^*(1+[\partial_sz]_\alpha)|c|^{1+\alpha},
\end{equation}
where
\begin{equation}\label{e:sigma}
\sigma(a)=\frac{1-a^2}{(1+a^2)^2}.
\end{equation}
\end{lemma}

\begin{proof}
Let us fix $s\in\R$. Observe that if $c(s)=0$, \eqref{e:sing2} is obvious; therefore we may assume in the following that $c(s)\neq 0$, without loss of generality $c(s)>0$. We may then change variables in the integral defining $T_{\Phi}f$ to obtain   
\begin{equation}\label{e:sing21}
T_{\Phi}f(s)=\frac{c(s)}{2\pi}\int_\R \frac{\partial_s f(s-c(s)\xi)-\partial_s f(s)}{c(s)\xi}\Phi_c(\xi,s)\,d\xi,
\end{equation}
where $\Phi_c(\xi,s):=\Phi(c(s)\xi,s)$. 
Note that 
\begin{equation*}
\begin{split}
\Psi &=K\left(Z_c+\frac{1}{\xi}\right)+K\left(Z_c-\frac{1}{\xi}\right)-2K\left(Z_c\right)\\
&=\frac{1}{\xi^2}\int_0^1\left[K''(Z_c+\tfrac{\tau}{\xi})+K''(Z_c-\tfrac{\tau}{\xi})\right](1-\tau)d\tau,
\end{split}
\end{equation*}
where we denoted $Z_c(\xi,s)=Z(c(s)\xi,s)$. Using that both $K$ and $K''$ are uniformly bounded, it follows that
\begin{equation}\label{e:tildephi1}
|\Phi_c|\leq C\min\left(1,\xi^{-2}\right).
\end{equation}
Furthermore, letting
$$
\Phi_0(\xi,s)=K\left(\partial_sz(s)+\frac{1}{\xi}\right)+K\left(\partial_sz(s)-\frac{1}{\xi}\right)-2K\left(\partial_sz(s)\right),
$$
we also obtain, using that both $K$ and $K''$ are uniformly Lipschitz, that
\begin{equation}\label{e:tildephi2}
|\Phi_c-\Phi_0|\leq C\min\left(1,\xi^{-2}\right)[\partial_s z]_{\alpha}(c\xi)^{\alpha}.
\end{equation}
Since $\xi\mapsto\Phi_0(\xi,s)$ is a rational function of $\xi$, its precise integral in $\xi$ may be calculated by elementary methods (as done in \cite{fsz:ipm}), leading to 
$$
\int_{\R}\left[K(a+\tfrac{1}{\xi})+K(a-\tfrac{1}{\xi})-2K(a)\right]\,d\xi=-\pi \frac{1-a^2}{(1+a^2)^2}=-\pi \sigma(a).
$$
Thus 
$\frac{1}{2\pi}\int_\R\Phi_0(\xi,s)\,d\xi=-\frac{1}{2}c(s)\sigma(\partial_sz(s))$. We then write \eqref{e:sing21} as
\begin{equation*}
\begin{split}
	T_{\Phi}f(s)=&\frac{1}{2}c(s)\sigma(\partial_sz(s))\partial_s^2f(s)+\\
	&+\frac{c(s)}{2\pi}\partial_s^2f(s)\int_{\R}\Phi_0(\xi,s)-\Phi_c(\xi,s)\,d\xi+\\
	&+\frac{c(s)}{2\pi}\int_{\R}\left[\frac{\partial_s f(s-c(s)\xi)-\partial_s f(s)}{c(s)\xi}+\partial_s^2f(s)\right]\Phi_c(\xi,s)\,d\xi.
\end{split}	
\end{equation*}
Using the bounds \eqref{e:tildephi1}-\eqref{e:tildephi2} on the two integrals we then deduce \eqref{e:sing2}.
\end{proof}

\bigskip

Finally, we turn our attention to $I_{ij}$ defined in \eqref{e:Iij}. 
\begin{lemma}\label{l:I1}
There exists a constant $C>1$ such that the following holds.
Let $z(s,t)=\beta s+\tilde z(s,t)$ with $\tilde z\in C([0,T];C^{1,\alpha}(\R))$, 
and let
$$
I(s,t)=PV\int_{\R}\frac{\xi}{\xi^2+(z(s,t)-z(s-\xi,t))^2}d\xi.
$$
Then $I\in C([0,T];C(\R))$ with
$$
|I(s,t)|\leq C\|\partial_sz\|_{\alpha}.
$$
\end{lemma}

\begin{proof}
Using that $PV\int_{|\xi|<1}\frac{d\xi}{\xi}=PV\int_{|\xi|>1}\frac{d\xi}{\xi}=0$ and using the notation introduced in \eqref{e:Zt}-\eqref{e:defK} we have
\begin{equation*}
\begin{split}
	I(s,t)&=PV\int_{\R}K(Z)\frac{d\xi}{\xi}\\
	&=\int_{|\xi|<1}[K(Z)-K(\partial_sz)]\frac{d\xi}{\xi}+\int_{|\xi|>1}[K(Z)-K(\beta)]\frac{d\xi}{\xi}.
\end{split}	
\end{equation*}
Observe that the integrands in these two integrals are uniformly integrable. Indeed,
$K:\R\to\R$ is uniformly Lipschitz continuous, so that for the first integral we may use $|Z-\partial_sz|\leq [\partial_sz]_{\alpha}|\xi|^\alpha$ and for the second integral $|Z-\beta|\leq 2\|\tilde z\|_{0}|\xi|^{-1}$. Thus
$$
\left|\frac{K(Z)-K(\partial_sz)}{\xi}\right|\leq [\partial_sz]_{\alpha}|\xi|^{\alpha-1},\quad \left|\frac{K(Z)-K(\beta)}{\xi}\right|\leq 2\|\tilde z\|_0|\xi|^{-2}.
$$	
The conclusion follows from Lebesgue's dominated convergence theorem.
\end{proof}

\begin{lemma}\label{l:I2}

There exists a constant $C>1$ such that the following holds.
Let $z(s)=\beta s+\tilde z(s)$ with $\tilde z\in C^{1,\alpha}(\R)$, $c\in C(\R)$ with $c(s)\neq 0$,
and let
\begin{equation*}
I_c(s)=PV\int_{\R}K\left(Z(\xi,s)+\tfrac{c(s)}{\xi}\right)\frac{d\xi}{\xi}.
\end{equation*}
Then 
$$
|I_c(s)|\leq C(1+\|\tilde z\|_{1,\alpha}+\|\partial_sz\|_0^2),
$$
and moreover
\begin{equation}\label{e:limIc}
\left|I_c(s)-PV\int_{\R} \left[K(\partial_sz+\tfrac{1}{\xi})+K(Z)\right]\frac{d\xi}{\xi}\right|\leq C[\partial_sz]_{\alpha}c(s)^{\alpha}.
\end{equation}
\end{lemma}

\begin{proof}
Let us fix $s\in\R$ and assume without loss of generality that $c(s)>0$.
We perform the change of variables $\xi\mapsto \frac{\xi}{c(s)}$ to obtain
\begin{equation*}
\begin{split}
	&I_c=PV\int_{\R}K\left(Z+\tfrac{c}{\xi}\right)\frac{d\xi}{\xi}=PV\int_{\R}K\left(Z_c+\tfrac{1}{\xi}\right)\frac{d\xi}{\xi}\\
	&=\int_{|\xi|>1}\left[K\left(Z_c+\tfrac{1}{\xi}\right)-K\left(Z_c\right)\right]\frac{d\xi}{\xi}+PV\int_{|\xi|>1}K\left(Z_c\right)\frac{d\xi}{\xi}+\int_{|\xi|<1}K\left(Z_c+\tfrac{1}{\xi}\right)\frac{d\xi}{\xi},
\end{split}
\end{equation*}
where we denoted, as in the proof of Lemma \ref{l:sing2}, $Z_c(\xi,s)=Z(c(s)\xi,s)$.
Now we 
observe the following: using the uniform Lipschitz bound on $K$,
\begin{equation*}
\left|K\left(Z_c+\tfrac{1}{\xi}\right)-K\left(Z_c\right)\right|\leq \frac{1}{|\xi|},	
\end{equation*}
so that the first integrand is dominated by $|\xi|^{-2}$ on $|\xi|>1$; furthermore, we have the lower bound
\begin{equation*}
\xi^2+(\xi Z_c+1)^2=(1+Z_c^2)\left(\xi+\tfrac{Z_c}{1+Z_c^2}\right)^2+\frac{1}{1+Z_c^2}\geq \frac{1}{1+\|\partial_sz\|_0^2},
\end{equation*}
implying that the third integrand is uniformly bounded. Finally, for the second integral we proceed as in the proof of Lemma \ref{l:I1}:
\begin{equation*}
\begin{split}
	PV&\int_{|\xi|>1}K\left(Z_c\right)\frac{d\xi}{\xi}=PV\int_{|\xi|>ct}K\left(Z\right)\frac{d\xi}{\xi}\\
	&=PV\int_{1>|\xi|>ct}K\left(Z\right)\frac{d\xi}{\xi}+PV\int_{|\xi|>1}K\left(Z\right)\frac{d\xi}{\xi}\\
	&=\int_{1>|\xi|>ct}[K(Z)-K(\partial_sz)]\frac{d\xi}{\xi}+\int_{|\xi|>1}[K(Z)-K(\beta)]\frac{d\xi}{\xi}.
\end{split}
\end{equation*}
Collecting the estimates above we obtain the uniform bound on $I_c$.

In order to prove \eqref{e:limIc}, we write
\begin{equation*}
\begin{split}
PV&\int_{\R} \left[K(Z+\tfrac{c}{\xi})-K(Z)-K(\partial_sz+\tfrac{1}{\xi})\right]\frac{d\xi}{\xi}=\\
=&PV\int_{\R} \left[K(Z_c+\tfrac{1}{\xi})-K(Z_c)\right]\frac{d\xi}{\xi}-PV\int_{\R}K(\partial_sz+\tfrac{1}{\xi}) \frac{d\xi}{\xi}\\
=&PV\int_{|\xi|>1} \left[K(Z_c+\tfrac{1}{\xi})-K(Z_c)-K(\partial_sz+\tfrac{1}{\xi})+K(\partial_sz)\right]\frac{d\xi}{\xi}\\
&+PV\int_{|\xi|<1}[K(\partial_sz)-K(Z_c)]\frac{d\xi}{\xi}+PV\int_{|\xi|<1}K(Z_c+\tfrac{1}{\xi})-K(\partial_sz+\tfrac{1}{\xi})\frac{d\xi}{\xi}.
\end{split}
\end{equation*}  
For the first term we use uniform boundedness of $K''$ to obtain the bound
\begin{equation*}
\begin{split}
&\left|K(Z_c+\tfrac{1}{\xi})-K(\partial_sz+\tfrac{1}{\xi})-K(Z_c)+K(\partial_sz)\right|=\frac{1}{|\xi|}\left|\int_0^1K'(Z_c+\tfrac{\tau}{\xi})-K'(\partial_sz+\tfrac{\tau}{\xi})d\tau\right|\\
&\leq C|\xi|^{-1}|Z_c-\partial_sz|\leq C[\partial_sz]_{\alpha}c^{\alpha}|\xi|^{\alpha-1},
\end{split}
\end{equation*}
which suffices to bound the integral. For the second and third term we obtain analogously
$$
\left|K(Z_c)-K(\partial_sz)\right|+\left|K(Z_c+\tfrac{1}{\xi})-K(\partial_sz+\tfrac{1}{\xi})\right|\leq |Z_c-\partial_sz|\leq [\partial_sz]_{\alpha}(c|\xi|)^{\alpha},
$$
which again allows to bound the integrals. The estimate \eqref{e:limIc} follows.
\end{proof}

\bigskip

With Lemmas \ref{l:reg}-\ref{l:I2} we are now in a position to obtain an expansion in time as $t\to 0$ for the normal velocities $u^{(i)}_\nu$ for sufficiently regular $z,c$.  

\begin{proposition}\label{p:expansion}
Let $z(s,t)=\beta s+\tilde z(s,t)$ with $\tilde z\in C^1([0,T];C^{1,\alpha}_*(\R))$, $c>0$ with $\partial_sc\in C^{\alpha}_*(\R)$, and define $c_i, \Phi_{ij}, z^{(i)}$ as in \eqref{e:ci}, \eqref{e:defPhiij} and  \eqref{e:defzj}. Then there exists a constant $C_{z,c}$ depending on $\|z\|_{1,\alpha}^*$, $\|c\|_{1,\alpha}^*$, and $N$, such that
\begin{equation}\label{e:expansion1}
\|u^{(i)}_\nu-T_{\Phi_0}z_0\|^*_0\leq C_{z,c}t^{\alpha},
\end{equation}
where $z_0(s)=z(s,0)$, $\Phi_0(\xi,s)=\Phi(\xi,s,0)$ and 
\begin{equation}\label{e:sharpinterface}
\Phi(\xi,s,t)=\frac{2\xi^2}{\xi^2+(z(s,t)-z(s-\xi,t))^2}.
\end{equation}
Furthermore
\begin{equation}\label{e:expansion2}
\left\|\frac{1}{2N}\sum_{i}u^{(i)}_\nu-\left(T_{\Phi_0}z_0+tT_{\Psi_0}z_0+tT_{\Phi_0}z_1+t\bar{c}\sigma(\partial_sz_0)\partial_s^2z_0\right)\right\|_0^*\leq C_{z,c}t^{1+\alpha},
\end{equation}
where $z_1(s)=\partial_tz(s,0)$, $\Psi_0(\xi,s)=\partial_t\Phi(\xi,s,0)$ and
\begin{equation}\label{e:barc}
\bar{c}(s)=\frac{1}{8N^2}\sum_{i,j}|c_i(s)-c_j(s)|.
\end{equation}
\end{proposition}

\begin{proof}
Using the representation \eqref{e:unui2} we write
\begin{equation*}
\begin{split}
u^{(i)}_\nu &=\frac{1}{N}\sum_{j}T_{\Phi_{ij}}z^{(j)}+t(\partial_sc_i-\partial_sc_j)I_{ij}\\
&= T_{\Phi_0}z_0+\frac{1}{N}\sum_{j}\left[T_{(\Phi_{ij}-\tfrac12\Phi_0)}z_0+T_{\Phi_{ij}}(z^{(j)}-z_0)+t(\partial_sc_i-\partial_sc_j)I_{ij}\right].
\end{split}
\end{equation*}
The estimate \eqref{e:expansion1} then follows from applying Lemma \ref{l:sing1}, Lemma \ref{l:sing0} and Lemma \ref{l:I2}.
Next, we calculate 
\begin{equation*}
\begin{split}
\sum_{i}u^{(i)}_\nu &=\frac{1}{N}\sum_{i,j}T_{\Phi_{ij}}z^{(j)}+t(\partial_sc_i-\partial_sc_j)I_{ij}\\	
&=\frac{1}{2N}\sum_{i,j}(T_{\Phi_{ij}}z^{(j)}+T_{\Phi_{ji}}z^{(i)})+t(\partial_sc_i-\partial_sc_j)(I_{ij}-I_{ji})\\
&=2N\,T_{\Phi_0}z_0+\frac{1}{2N}\sum_{i,j}T_{(\Phi_{ij}+\Phi_{ji}-\Phi_0)}z_0+\frac{1}{N}\sum_{i,j}T_{\Phi_{ij}}(z^{(j)}-z_0)\\
&+\frac{1}{2N}\sum_{i,j}t(\partial_sc_i-\partial_sc_j)(I_{ij}-I_{ji})
\end{split}
\end{equation*}
Now, let us write
\begin{equation*}
\begin{split}
\Phi_{ij}+\Phi_{ji}-\Phi_0=&\left[K(Z^{(j)}+\tfrac{c_{ij}t}{\xi})-K(Z_0+\tfrac{c_{ij}t}{\xi})-tK'(Z_0)Z^{(j)}_1\right]+\\
+&\left[K(Z^{(i)}-\tfrac{c_{ij}t}{\xi})-K(Z_0-\tfrac{c_{ij}t}{\xi})-tK'(Z_0)Z^{(i)}_1\right]+\\
+tK'(Z_0)(Z^{(i)}_1+&Z^{(j)}_1)+\left[K(Z_0+\tfrac{c_{ij}t}{\xi})+K(Z_0-\tfrac{c_{ij}t}{\xi})-2K(Z_0)\right],
\end{split}
\end{equation*}
where $Z_0(\xi,s)=\frac{z_0(s)-z_0(s-\xi)}{\xi}$, $Z_1^{(i)}(\xi,s)=\partial_t|_{t=0}Z^{(i)}(\xi,s,t)$ and $c_{ij}=c_i-c_j$.
Recalling the convention that $c_{-i}=-c_i$, we see that 
$$
\sum_{i,j}K'(Z_0)(Z^{(i)}_1+Z^{(j)}_1)=8N^2K'(Z_0)Z_1=4N^2\Psi_0.
$$
On the first two terms we can argue as in Lemma \ref{l:sing0} and Lemma \ref{l:sing1}, whereas Lemma \ref{l:sing2} applies to the last term. 
Finally, Lemma \ref{l:I2} implies that
$$
|I_{ij}-I_{ji}|\leq Ct^{1+\alpha}.
$$
This concludes the proof of \eqref{e:expansion2}.
\end{proof}

\begin{remark}
In the statement of Proposition \ref{p:expansion} we left the expression for $\bar{c}$ in the general form \eqref{e:barc}, which is valid for any choice of $0<c_1(s)<c_2(s)<\dots<c_N(s)$ with $c_{-i}(s)=-c_i(s)$ and $\partial_sc_i\in C^\alpha_*(\R)$. The explicit choice of $c_i(s)$ in \eqref{e:ci} leads to the following simplified expression: 
\begin{equation*}
\begin{split}
\bar{c}(s)&=\tfrac{1}{4N^2}\sum_{i,j=1}^N|c_i(s)-c_j(s)|+|c_i(s)+c_j(s)|\\
&=\tfrac{1}{2N^2}\sum_{i,j=1}^N\max(c_i(s),c_j(s))=\tfrac{c(s)}{2N^2(2N-1)}\sum_{i,j=1}^N(2\max(i,j)-1)\\
&=\frac{2N+1}{3N}c(s).
\end{split}
\end{equation*}
\end{remark}

\begin{remark}\label{r:expansion}
It is instructive to compare the expansions in Proposition \ref{p:expansion} with the expansion of the sharp interface evolution in \eqref{evolz}. In particular, recall that the right hand side of \eqref{evolz} is the normal component of the velocity at the interface. Let us denote this by $v_{\nu}$, so that
$$
v_{\nu}=T_{\Phi}z,
$$
with $\Phi$ defined in \eqref{e:sharpinterface}. In view of Lemma \ref{l:reg} $v_{\nu}\in C^2([0,T];C^{\alpha}_*(\R))$ and therefore \eqref{e:expansion1},\eqref{e:expansion2} amount to the expansions
\begin{equation*}
\begin{split}
	u^{(i)}_{\nu}&=v_{\nu}+O(t^{\alpha}),\\
	\tfrac{1}{2N}\sum_iu^{(i)}_{\nu}&=v_{\nu}+t\bar{c}\sigma(\partial_sz_0)\partial_s^2z_0+O(t^{1+\alpha}).
\end{split}
\end{equation*}

\end{remark}


\section{Construction of the curve $z$}\label{s:z}

In this section we construct a function $z=z(s,t)$ satisfying the conditions of Proposition \ref{p:SubsN}. 

\begin{proposition}\label{p:z}
Let $z_0(s)=\beta s+\tilde{z}_0(s)$ with $\tilde{z}_0\in C^{3,\alpha}_*(\R)$ for some $0<\alpha<1$ and $\beta\in\R$. Let $c=c(s)>0$ with $\partial_sc\in C^{\alpha}_*(\R)$. If $\inf_sc(s)=0$, assume in addition that $\beta=0$ and there exists $c_{min}>0$ such that 
\begin{equation}\label{e:cinf}
c(s)\geq c_{min}(1+|s|^{2\alpha/3})^{-1}.
\end{equation}
For any $T>0$ there exists $\tilde z\in C^2([0,T];C^{1,\alpha}_*(\R))$ such that the conditions \eqref{z1N}-\eqref{z3N} of Proposition \ref{p:SubsN} are satisfied.  
\end{proposition}

\begin{proof}
First of all, let us fix $f_1,f_2\in C_c^\infty(\R)$ such that $\int_{I_k}f_l(s)\,ds=\delta_{kl}$, where $I_1=(-\infty,0)$ and $I_2=(0,\infty)$.
We define
\begin{equation}\label{e:definez}
	z(s,t)=z_0(s)+tz_1(s)+\tfrac{1}{2}t^2z_2(s)+\sum_{k=1}^2\psi_k(t)f_k(s),
\end{equation}
where
\begin{align*}
	z_1&:=T_{\Phi_0}\tilde{z}_0,\\
	z_2&:=T_{\Phi_0}z_1+T_{\Psi_0}\tilde{z}_0+\bar c\sigma(\partial_sz_0)\partial_s^2z_0,
\end{align*}
with $\Phi_0$, $\Psi_0$ and $\bar{c}$ as defined in Proposition \ref{p:expansion}, and
 $\psi_k\in C^2([0,T])$ are functions of time still to be fixed, such that $\psi_k(0)=\psi_k'(0)=\psi_k''(0)=0$ for $k=1,2$.

Let us check that $z$ satisfies the conditions of Proposition \ref{p:expansion}. Since $\tilde{z}_0\in C^{3,\alpha}_*(\R)$, Lemma \ref{l:reg} (applied with $z(s,t)=z_0(s)$) implies that $\Phi_0\in \mathcal{W}^{2,\alpha}$. Theorem \ref{t:TPhi} then implies that $z_1\in C^{2,\alpha}_*(\R)$. Consequently, from the expression for $\Psi_0$ in Proposition \ref{p:expansion} and using Lemma \ref{l:reg} again we deduce that $\Phi_1\in \mathcal{W}^{1,\alpha}$. Hence $z_2\in C^{1,\alpha}_*(\R)$, using once more Theorem \ref{t:TPhi}. Thus we have shown that $z\in C^2([0,T];C^{1,\alpha}_*(\R))$. It follows now from Proposition \ref{p:expansion} that
\begin{eqnarray}
\lim_{t\to 0}\sup_{s}(1+|s|^{1+\alpha})\left|\partial_tz(s,t)-u_{\nu}^{(i)}(s,t)\right|&=  0\,, \label{zz1}\\
\lim_{t\to 0}\sup_s\frac{1}{t}(1+|s|^{1+\alpha})\left|\partial_tz(s,t)-\frac{1}{2N}\sum_{i}u_\nu^{(i)}(s,t)\right| &=0\,. \label{zz2}
\end{eqnarray} 
In the case $\inf_s c(s)>0$ conditions \eqref{z1N}-\eqref{z3N} follow directly from \eqref{zz1}-\eqref{zz2}. In particular, in this case we can take $\psi_k\equiv 0$ for $k=1,2$. 

In what follows, let us then assume $\inf_s c(s)=0$, so that also $\beta=0$ and \eqref{e:cinf} holds. Observe first of all that \eqref{e:cinf} together with \eqref{zz1} directly implies \eqref{z1N}. Furthermore, 
$$
|\partial_sz|,|\partial_sc|\lesssim (1+|s|^{1+\alpha})^{-1}\lesssim c^{1/2},
$$
so that also \eqref{z3N} holds. It remains to verify \eqref{z2N}. To this end our aim is to choose $\psi_k$ in such a way that 
\begin{equation}\label{e:integral}
	\int_{I_k}\partial_tz(s,t)-\frac{1}{2N}\sum_{i}u_\nu^{(i)}(s,t)\,ds=0\quad\forall\,t\in[0,T], k=1,2.
\end{equation}
Indeed, assume for the moment that \eqref{e:integral} holds. Then \eqref{zz2} implies for $s>0$ 
$$
\int_0^s\partial_tz-\frac{1}{2N}\sum_{i}u_\nu^{(i)}\,ds'=-\int_s^\infty\partial_tz-\frac{1}{2N}\sum_{i}u_\nu^{(i)}\,ds'=\frac{o(t)}{1+|s|^{\alpha}}\sim o(t)c^{3/2}(s),
$$
and similarly for $s<0$. Thus in this case also \eqref{z2N} is verified. 

To complete the proof it therefore remains to choose $\psi_k$ to ensure \eqref{e:integral}. Let 
$$
h_k(t,\psi(t)):=\int_{I_k}\frac{1}{2N}\sum_{i}u_\nu^{(i)}(s',t)-z_1(s')-tz_2(s')\,ds',
$$
where dependence on $\psi=(\psi_1,\psi_2)$ appears in the implicit dependence of $u_\nu^{(i)}$ on $z$ defined in \eqref{e:definez} via \eqref{e:unui}-\eqref{e:defzj}. Then, recalling the choice of $f_k$, \eqref{e:integral} is equivalent to the ODE system
$$
\psi_k'(t)=h_k(t,\psi(t)),\quad k=1,2
$$
with initial condition $\psi_k(0)=0$. Using estimate \eqref{e:W02} in Lemma \ref{l:sing0} we verify that $x\mapsto h(t,x)$ is uniformly Lipschitz continuous. Therefore the Cauchy-Lipschitz theorem is applicable and yields a unique solution $\psi:[0,T]\to \R^2$. 

Then, by recalling our choice for $z_1$, $z_2$ and the expansion \eqref{e:expansion2} we see that $h(t,0)=o(t)$. By differentiating the ODE we also obtain $\tfrac{d}{dt}\psi'(t)=\partial_th+\partial_xh\psi'$. Since $\psi(0)=0$ and $h(0,0)=0$, we obtain $\psi'(0)=0$, hence $\psi(t)=o(t)$. Furthermore, since $\partial_t h(0,0)=0$ also, we deduce $\psi''(0)=0$, and furthermore, from the equation we deduce $\psi'(t)=o(t)$. This completes the proof.    

\end{proof}


\bibliographystyle{acm}


\end{document}